\documentclass[11pt,hyp,]{nyjm}
\usepackage{hyperref}
\hypersetup{nesting=true,debug=true,naturalnames=true}
\usepackage{graphicx,amssymb,upref}

\let\<\langle
\let\>\rangle
\newcommand{\doi}[1]{doi:\,\href{http://dx.doi.org/#1}{#1}}

\let\uml\"

\title[On the number of ramified primes]{On the number of ramified primes in specializations of function fields over $\Qq$}

\author{Lior Bary-Soroker}
\address{School of Mathematical Sciences, Tel Aviv University, Ramat Aviv, Tel Aviv 6997801, Israel}
\email{barylior@post.tau.ac.il}

\author{Fran\c cois Legrand}
\address{School of Mathematical Sciences, Tel Aviv University, Ramat Aviv, Tel Aviv 6997801, Israel}
\address{Department of Mathematics and Computer Science, the Open University of Israel, Ra'anana 4353701, Israel}
\email{flegrand@post.tau.ac.il}

\thanks{The first author is partially supported by the Israel Science Foundation (grant No. 40/14). The second author is partially supported by the Israel Science Foundation (grants No. 40/14 and No. 696/13).} 



\newcommand{\Zz}{\mathbb{Z}}
\newcommand{\Qq}{\mathbb{Q}} 
\usepackage{enumitem}
{\theoremstyle{plain}
\newtheorem{theorem}{Theorem}[section]    
\newtheorem{lemma}[theorem]{Lemma}       
\newtheorem{proposition}[theorem]{Proposition}  
\newtheorem{corollary}[theorem]{Corollary}   
}
{\theoremstyle{remark}
\newtheorem{definition}[theorem]{Definition}      
\newtheorem{remark}[theorem]{Remark}   
}

\begin{document} 
 
\begin{abstract}  
We study the number of ramified prime numbers in finite Galois extensions of $\mathbb{Q}$ obtained by specializing a finite Galois extension of $\mathbb{Q}(T)$. Our main result is a central limit theorem for this number. We also give some Galois theoretical applications.
\end{abstract} 

\maketitle

\tableofcontents

\section{Introduction}

Given an indeterminate $T$, the {\it{specialization}} of a finite Galois extension $E/\Qq(T)$ with Galois group $G$ at a point $t_0 \in \mathbb{P}^1(\Qq)$, which is not a branch point, is a finite Galois extension of $\Qq$ whose Galois group is a subgroup of $G$; we denote it by $E_{t_0}/\Qq$ (see \S\ref{2.1} for basic terminology). For example, if $E$ is the splitting field over $\Qq(T)$ of a monic polynomial $P(T,Y) \in \Qq[T][Y]$ which is separable in $Y$, then, $E_{t_0}$ is the splitting field over $\Qq$ of $P(t_0,Y)$ (for all but finitely many $t_0\in \Qq$).

\subsection{The arithmetic function ${\rm{Ram}}_{E/\mathbb{Q}(T)}$}

In this paper, we are interested in the number of prime numbers ramifying in finite Galois extensions of $\Qq$ obtained by specializing a finite Galois extension of $\mathbb{Q}(T)$ at positive integers. More precisely, let us define:

\begin{definition}
Let $E/\Qq(T)$ be a finite Galois extension. Given a positive integer $n$ which is not a branch point, let 
$${\rm{Ram}}_{E/\mathbb{Q}(T)}(n)$$ be the number of ramified prime numbers in the specialization $E_n/\Qq$. If $n$ is a branch point, we set arbitrarily ${\rm{Ram}}_{E/\mathbb{Q}(T)}(n) = -1$.
\end{definition}

\noindent
Note that ${\rm{Ram}}_{E/\mathbb{Q}(T)}$ depends on the choice of the indeterminate $T$.

\begin{remark} \label{rem 1.2}
If $E/\Qq(T)$ is trivial over $\overline{\Qq}$ \footnote{i.e., if the compositum of $E$ and $\overline{\Qq}(T)$ (in a given algebraic closure of $\Qq(T)$) is $\overline{\Qq}(T)$ or, equivalently, if there exists a number field $F$ such that $E=F(T)$.}, then, there are no branch points and the extension $E_{t_0}/\Qq$ does not depend on $t_0$. In particular, the function ${\rm{Ram}}_{E/\mathbb{Q}(T)}$ is constant. {\it{Hence, we tactically assume throughout this paper that the extension $E/\Qq(T)$ is not trivial over $\overline{\Qq}$.}}
\end{remark}

Some properties of the function ${\rm{Ram}}_{E/\mathbb{Q}(T)}$ can be derived from results in the literature. For example, it is unbounded. More precisely, the second author \cite{Leg16, Leg17} proves that, given a finite Galois extension $E/\Qq(T)$ with Galois group $G$ and a finite set $\mathcal{S}$ of sufficiently large suitable prime numbers (depending on the extension $E/\Qq(T)$), there exist infinitely many positive integers $n$ such that the specialization of $E/\Qq(T)$ at $n$ has Galois group $G$ and ramifies at each prime number of $\mathcal{S}$ \footnote{Actually the inertia groups at prime numbers in $\mathcal{S}$ in the specializations can be prescribed and explicit bounds on their discriminants are given.}. In particular, given a positive integer $m$, there exist infinitely many positive integers $n$ such that ${\rm{Gal}}(E_n/\Qq)=G$ and ${\rm{Ram}}_{E/\Qq(T)}(n) \geq m$. 

On the other hand, the first author and Schlank \cite{BSS16} prove that the function ${\rm{Ram}}_{E/\mathbb{Q}(T)}$ does not tend to $\infty$. Furthermore, several works consist in producing, for some finite groups $G$ and some specific finite Galois extensions $E/\Qq(T)$ with Galois group $G$, some positive integers $n$ such that the specialization $E_n/\Qq$ has Galois group $G$ and the number ${\rm{Ram}}_{E/\Qq(T)}(n)$ is small; see, e.g., \cite{JR03, MR05, JR07, Rob11, BSS16}. For example, for $G=S_N$ ($N \geq 3$) and some specific realizations over $\mathbb{Q}(T)$ of $S_N$ with 3 branch points, one has ${\rm{Ram}}_{E/\Qq(T)}(n) \leq 3$ for infinitely many positive integers $n$; see \cite{BSS16} (in \emph{loc.cit.} the infinite prime is also counted).

\subsection{Main result} 

We study the statistical properties of the arithmetic function ${\rm{Ram}}_{E/\Qq(T)}$ for a given finite Galois extension $E/\Qq(T)$. 

Recall that the absolute Galois group of $\Qq$ acts on the branch points of the extension $E/\Qq(T)$ lying in $\overline{\Qq}$ (i.e., which are different from $\infty$). Let $r$ be the number of orbits under this action. By the Riemann-Hurwitz formula, one has $r \geq 1$ (as the extension $E/\Qq(T)$ has been assumed not to be trivial over $\overline{\Qq}$; see Remark \ref{rem 1.2}).

\begin{theorem} \label{BSL}
For each positive integer $k$, one has
$$\lim\limits_{\substack{N \to \infty }} \frac{1}{N} \sum_{0<n \leq N} \bigg(\frac{{\rm{Ram}}_{E/\mathbb{Q}(T)}(n) - r \log \log(N)}{\sqrt{r \log \log (N)}} \bigg)^k = \frac{1}{\sqrt{2 \pi}} \int_{- \infty}^{+\infty} t^k e^\frac{-t^2}{2} \, dt.$$
\end{theorem}

\noindent
Although ${\rm{Ram}}_{E/\Qq(T)}$ depends on the choice of $T$, the limit distribution of the normalization of ${\rm{Ram}}_{E/\Qq(T)}$ given in Theorem \ref{BSL} does not.

Taking $k=1$ and $k=2$ in Theorem \ref{BSL} gives the following:
$$\frac{1}{N}\sum_{0<n \leq N} {\rm{Ram}}_{E/\mathbb{Q}(T)}(n) \, \, {\sim} \, \, r \log \log (N), \quad N \to \infty,$$
$$\frac{1}{N} \sum_{0<n \leq N} \big({\rm{Ram}}_{E/\mathbb{Q}(T)}(n) - r \log \log (N) \big)^2 \, \, {\sim} \, \, r \log \log (N) , \quad N \to \infty.$$

Moreover, by the method of moments (see, e.g., \cite[Example 30.1 and Theorem 30.2]{Bil95}), Theorem \ref{BSL} provides the limit distribution of our normalization of ${\rm{Ram}}_{E/\Qq(T)}$. 

For every real number $a$, set $$I(a)=\frac{1}{\sqrt{2 \pi}} \int_{- \infty}^{a} e^\frac{-t^2}{2} \, dt.$$

\begin{theorem} \label{BSL bis}
For every real number $a$, one has $$\lim\limits_{\substack{N \to \infty }} \frac{1}{N} \left \vert \left \{0<n \leq N \, : \, \frac{{\rm{Ram}}_{E/\mathbb{Q}(T)}(n) - r \log \log(N)}{\sqrt{r \log \log (N)}} \leq a \right \} \right \vert = I(a).$$
\end{theorem}

\noindent
Similar results hold for finite extensions $E/\Qq(T)$ which are not necessarily Galois since, in this case, ${\rm{Ram}}_{E/\Qq(T)}={\rm{Ram}}_{\widehat{E}/\Qq(T)}$, with $\widehat{E}$ the Galois closure of $E$ over $\Qq(T)$ (see $\S$\ref{5}).

\subsection{Applications}

Below, we give three corollaries of Theorem \ref{BSL} (see $\S$\ref{3} for the proofs).

\subsubsection{Application to inverse Galois theory}

A classical motivation to study specializations of finite Galois extensions of $\Qq(T)$ is the {\it{inverse Galois problem}}: does every finite group $G$ occur as the Galois group of a Galois extension of $\mathbb{Q}$? Indeed, a way to realize $G$ is by specializing a Galois extension $E/\mathbb{Q}(T)$ with Galois group $G$: from the {\it{Hilbert irreducibility theorem}}, there exist infinitely many positive integers $n$ each of which satisfies the {\it{Hilbert specialization property}}, i.e., such that the specialization $E_n/\Qq$ still has Galois group $G$. Many finite groups have been shown to occur as a Galois group over $\Qq$ by this method; we refer to \cite{MM99} for more details and references, and to \cite{Zyw14} for more recent results.

We show that Theorem \ref{BSL} still holds if we restrict to the set of positive integral specialization points which satisfy the Hilbert specialization property:

\begin{corollary} \label{coro 1}
Denote the Galois group of the extension $E/\Qq(T)$ by $G$. Then, for each positive integer $k$, one has
$$\lim\limits_{\substack{N \to \infty }} \frac{1}{N} \hspace{-1.5mm} \sum_{\substack{{0<n \leq N} \\  {\rm{Gal}}(E_n/\Qq)= G}} \hspace{-2mm}\bigg(\frac{{\rm{Ram}}_{E/\mathbb{Q}(T)}(n) - r \log \log(N)}{\sqrt{r \log \log (N)}} \bigg)^k = \frac{1}{\sqrt{2 \pi}} \int_{- \infty}^{+\infty} \hspace{-2mm} t^k e^\frac{-t^2}{2} \, dt.$$
\end{corollary}

Taking $k=1$ in Corollary \ref{coro 1} gives the following:
$$\frac{1}{N}\sum_{\substack{{0<n \leq N} \\  {\rm{Gal}}(E_n/\Qq)= G}} {\rm{Ram}}_{E/\mathbb{Q}(T)}(n) \, \, \sim \, \,  r \log \log (N) , \quad N\to \infty.$$
Hence, we reobtain that, given an integer $m \geq 1$, there exist integers $n \geq 1$ such that ${\rm{Gal}}(E_n/\Qq)=G$ and ${\rm{Ram}}_{E/\Qq(T)}(n) \geq m$. In particular, if a given non-trivial finite group $G$ occurs as the Galois group of a finite Galois extension of $\Qq(T)$ which is not trivial over $\overline{\Qq}$, then, {\it{given a positive integer $m$, there exists a finite Galois extension of $\Qq$ with Galois group $G$ and at least $m$ ramified prime numbers.}} We notice that, for some Galois groups over $\Qq$, the latter condition has not been proved yet. For example, there exist odd prime numbers $p$ for which all known realizations of ${\rm{PSL}}_2(\mathbb{F}_p)$ over $\Qq$ ramify only at $2$ and $p$ \cite{Zyw15}.

\subsubsection{Two corollaries on the function ${\rm{Ram}}_{E/\Qq(T)}$} \label{1.3}

From Theorem~\ref{BSL} with $k=2$, we get a normal order of the function ${\rm{Ram}}_{E/\Qq(T)}$:

\begin{corollary} \label{coro 2}
Let $\epsilon >0$. Then, for each positive integer $n$ which is not in some set $S_\epsilon$ which has asymptotic density zero, one has
$$(1-\epsilon) \cdot r \log  \log (n) \leq {\rm{Ram}}_{E/\Qq(T)} (n) \leq (1+\epsilon) \cdot r  \log  \log (n).$$
\end{corollary}

Consequently, the set of all positive integers $n$ such that ${\rm{Ram}}_{E/\Qq(T)} (n)$ $\leq C$ for a given non-negative integer $C$ has asymptotic density zero. The following corollary, which rests on Theorem \ref{BSL} with arbitrary $k$, gives upper bounds on the rate of convergence.

\begin{corollary} \label{coro 3}
Let $C$ and $k$ be two non-negative integers with $k \geq 1$. Then, there are some positive constants $\alpha(k,r)$ and $A(C,k,r)$ such that
$$\frac{1}{N} \bigg| \bigg \{0 < n \leq N \, : \, {\rm{Ram}}_{E/\Qq(T)}(n) \leq C \bigg \} \bigg| \leq \frac{\alpha(k,r)}{\log \log (N) ^k}$$
for each positive integer $N \geq A(C,k,r)$.
\end{corollary}

\subsection{Summary of the proof of Theorem \ref{BSL}} \label{1.4}

The proof, given in \S\ref{4}, has two parts we summarize below. Let $P_E(T) \in \Zz[T]$ be a separable polynomial whose roots are the finite branch points of $E/\Qq(T)$.

First, given a positive integer $n$ which is not a branch point of the extension $E/\Qq(T)$, we relate the number ${\rm{Ram}}_{E/\mathbb{Q}(T)}(n)$ to the number $\omega(P_E(n))$ of distinct prime numbers dividing $P_E(n)$ (without multiplicity). Namely, we make the difference  $${\rm{Ram}}_{E/\mathbb{Q}(T)}(n) - \omega(P_E(n))$$
completely explicit up to $O(1)$ (Lemma \ref{lemma 5}). This step is based on the use of a classical result about ramification in specializations \cite{Bec91}, \cite{Con00}, \cite[\S3.2]{Leg17} (see Lemma \ref{lemma 3}) and of some generalized version of the arithmetic function $\omega$ (Definition \ref{def 4}).

Next, we study this prime divisor counting function (Lemma \ref{lemma 6}) and then show that the difference ${\rm{Ram}}_{E/\mathbb{Q}(T)}(n) - \omega (P_E(n))$ is negligible in our context. Namely, for each positive integer $k$, we show that
\begin{equation} \label{4.6 intro}
\sum_{0<n \leq N} \Big({\rm{Ram}}_{E/\Qq(T)}(n) - \omega(P_E(n)) \Big)^k \, \, {=} \, \, O(N)
\end{equation}
as $N$ tends to $\infty$ (Lemma \ref{lemma 7}). By a result of Halberstam \cite[Theorem 4]{Hal56}\footnote{which generalizes the so-called Erd\H{o}s-Kac theorem \cite{EK40} on the Gaussian behaviour of the number of prime divisors of an integer. See \cite{GS07} for a simple proof of the Erd\H{o}s-Kac theorem and a review of the literature on this result.}, one has 
\begin{equation} \label{Hal intro}
\lim\limits_{\substack{N \to \infty }} \frac{1}{N} \hspace{-1mm} \sum_{0<n \leq N} \hspace{-1mm} \bigg(\frac{\omega(P_E(n)) - r \log \log(N)}{\sqrt{r \log \log (N)}} \bigg)^k \hspace{-1mm} = \hspace{-0.5mm} \frac{1}{\sqrt{2 \pi}} \int_{- \infty}^{+\infty} t^k e^\frac{-t^2}{2} \, dt.
\end{equation}
Conjoining \eqref{4.6 intro} and \eqref{Hal intro} then provides Theorem \ref{BSL}.

\vspace{2mm}

{\bf{Acknowledgments.}} We wish to thank Pierre D\`ebes, Steve Lester, and Z\'eev Rudnick for helpful discussions and valuable comments.

\section{Preliminaries and notation} \label{2}

\subsection{Preliminaries} \label{2.1} Let $T$ be an indeterminate and $E/\Qq(T)$ a finite Galois extension, assumed not to be trivial over $\overline{\Qq}$.

A point $t_0 \in \mathbb{P}^1(\overline{\Qq})$ is a {\it{branch point}} of $E/\Qq(T)$ if the prime ideal $(T-t_0) \, \overline{\Qq}[T-t_0]$ \footnote{Replace $T-t_0$ by $1/T$ if $t_0 = \infty$.} ramifies in the integral closure of $\overline{\Qq}[T-t_0]$ in the compositum of $E$ and $\overline{\Qq}(T)$ (in a fixed algebraic closure of ${\Qq}(T)$). The extension $E/\Qq(T)$ has only finitely many branch points and their number is positive (actually at least 2); see Remark \ref{rem 1.2}.

Given a point $t_0 \in \mathbb{P}^1(\Qq)$ which is not a branch point, the residue field of a prime ideal $\mathcal{P}$ lying over $(T-t_0) \, {\Qq}[T-t_0]$ in the extension ${E}/\Qq(T)$ is denoted by ${E}_{t_0}$ and we call the extension ${E}_{t_0}/\Qq$ the {\it{specialization}} of ${E}/\Qq(T)$ at $t_0$. This does not depend on the choice of the prime ideal $\mathcal{P}$  lying over $(T-t_0) \, {\Qq}[T-t_0]$ since ${E}/\Qq(T)$ is Galois. The specialization $E_{t_0}/\Qq$ is a Galois extension of $\Qq$ whose Galois group is a subgroup of ${\rm{Gal}}(E/\Qq(T))$, namely the decomposition group of the extension ${E}/\Qq(T)$ at $\mathcal{P}$.

\subsection{Notation} The notation below will be used throughout the paper.

Let $T$ be an indeterminate and $E/\Qq(T)$ a finite Galois extension with Galois group $G$. Recall that the absolute Galois group of $\Qq$ acts on the branch points of the extension $E/\Qq(T)$ lying in $\overline{\Qq}$. Let $r \geq 1$ be the number of distinct orbits under this action and 
\begin{equation}
\{t_1, \dots, t_r\}
\label{def:bp}
\end{equation}
a set of representatives. For each $i \in \{1,\dots,r\}$, denote the ramification index of $(T-t_i) \overline{\Qq}[T-t_i]$ in $E\overline{\Qq}/\overline{\Qq}(T)$ by 
\begin{equation}
e_i
\label{def:ri}
\end{equation}
and let 
\begin{equation}
P_i(T) \in \Zz[T]
\label{def:pol}
\end{equation}
be the unique polynomial with positive leading coefficient $b_i$, which is irreducible over $\mathbb{Z}$, and which satisfies $P_i(t_i)=0$. Finally, set
\begin{equation}
P_E(T)=\prod_{i=1}^r P_i(T).
\label{def:P_E}
\end{equation}

Denote by $\omega(n)$ the number of distinct prime divisors (without multiplicity) of a given positive integer $n$.

\section{Proofs of Corollaries \ref{coro 1}, \ref{coro 2}, and \ref{coro 3} assuming Theorem~\ref{BSL}} \label{3}

\subsection{Proof of Corollary \ref{coro 1}} \label{3.1}

We need first the following elementary bound. The lemma below will be used again in the last part of the proof of Theorem \ref{BSL} (\S\ref{4.2}).

\begin{lemma} \label{lemma 1}
One has ${\rm{Ram}}_{E/\Qq(T)}(n)  = O(\log(n)/\log \log(n))$, $n \to \infty$.
\end{lemma}

\begin{proof}
Let $P(T,Y) \in \mathbb{Z}[T][Y]$ be a monic separable (in $Y$) polynomial with splitting field $E$ over $\Qq(T)$ and $\Delta(T) \in \mathbb{Z}[T]$ its discriminant. For every integer $n$ which is not a root of $\Delta(T)$, $n$ is not branch point of $E/\Qq(T)$, the field $E_n$ is the splitting field over $\Qq$ of the polynomial $P(n,Y)$, and each prime number $p$ which ramifies in the extension $E_n/\Qq$ divides $\Delta(n)$. Hence, from the classical bound
$$\omega(n)=O(\log(n)/\log \log(n)), \quad n \to \infty$$ 
(see, e.g., \cite[$\S$V.15]{SMC06}) and as $\Delta(n)$ is polynomial in $n$, one gets
$${\rm{Ram}}_{E/\Qq(T)}(n) \leq \omega (\Delta(n)) =O(\log(n)/\log \log(n)), \quad n \to \infty,$$                          
as needed.
\end{proof}

\begin{proof} [Proof of Corollary \ref{coro 1}]
For any positive integers $k$ and $N$, set $$f_k(N)=\sum_{\substack{{0<n \leq N} \\  {\rm{Gal}}(E_n/\Qq) < G}} \bigg(\frac{{\rm{Ram}}_{E/\mathbb{Q}(T)}(n) - r \log \log(N)}{\sqrt{r \log \log (N)}} \bigg)^k.$$
By Theorem \ref{BSL}, it suffices to show $f_k(N)= o(N)$, $N \to \infty$, $k \geq 1$.

By Lemma \ref{lemma 1}, one has $f_k(N) = O(g(N) \cdot \log^k(N) \cdot (\log \log(N))^{-k})$, as $N$ tends to $\infty$, where $g(N)$ denotes the number of all positive integers $n \leq N$ such that ${\rm{Gal}}(E_n/\Qq) < G$. It then remains to use that 
$g(N)= O(\sqrt{N})$ as $N$ tends to $\infty$ (see, e.g., \cite[page 26]{Ser92}) to finish the proof.
\end{proof}

\subsection{Proof of Corollary \ref{coro 2}} \label{3.2}

Given a positive real number $\epsilon$, let $S_\epsilon$ be the set of all positive integers $n$ such that $$|{\rm{Ram}}_{E/\Qq(T)}(n) - r \log \log (n)| > \epsilon \cdot r \log \log (n).$$
Given a positive integer $N$, one has $$\frac{|\{0 <n \leq N \, : \, n \in S_\epsilon \}|}{N} \leq  \frac{1}{\sqrt{N}} + \frac{1}{N} \sum_{\substack{{\sqrt{N}<n \leq N} \\ n \in S_\epsilon}} 1.$$

\noindent
Then, to get Corollary \ref{coro 2}, it suffices to prove
\begin{equation} \label{Sepsilon0}
\frac{1}{N} \sum_{\substack{{\sqrt{N}<n \leq N} \\ n \in S_\epsilon}} 1 = o(1), \quad N \to \infty.
\end{equation}

By the definition of the set $S_\epsilon$, one has
\begin{equation} \label{Sepsilon}
\frac{1}{N} \sum_{\substack{{\sqrt{N}<n \leq N} \\ n \in S_\epsilon}} 1 < \frac{1}{N} \sum_{\sqrt{N}<n \leq N} \frac{({\rm{Ram}}_{E/\Qq(T)}(n) - r \log \log (n))^2}{\epsilon^2 \cdot (r \log \log(\sqrt{N}))^2}.
\end{equation}
As $(A-B)^2 \leq 2A^2+2B^2$ for any real numbers $A$ and $B$, we get
$$ ({\rm{Ram}}_{E/\Qq(T)}(n) - r \log \log (n))^2 \leq 2 \cdot ({\rm{Ram}}_{E/\Qq(T)}(n) - r \log \log (N))^2 $$
$$\hspace{21mm} + \, 2 r^2 \log^2(2)$$
for $\sqrt{N}<n \leq N$. Hence, the right-hand side in \eqref{Sepsilon} is smaller than
$$o(1) + \frac{2}{\epsilon^2 \cdot (r \log \log(\sqrt{N}))^2} \cdot \frac{1}{N} \sum_{0 <n \leq N} ({\rm{Ram}}_{E/\Qq(T)}(n) - r \log \log (N))^2.$$
By the case $k=2$ in Theorem \ref{BSL}, one has
$$\frac{1}{N} \sum_{0 <n \leq N} ({\rm{Ram}}_{E/\Qq(T)}(n) - r \log \log (N))^2 \sim r\log \log(N), \quad N \to \infty.$$
Hence, \eqref{Sepsilon0} holds and Corollary \ref{coro 2} follows. $\hfill \square$

\subsection{Proof of Corollary \ref{coro 3}} \label{3.3}

We shall need Lemma \ref{lemma 3.3} below whose proof is almost identical to the proof of Corollary \ref{coro 2}. The difference is that one applies Theorem \ref{BSL} with an arbitrary even integer $k$, in contrast to $k=2$.

Set
$$I_k=\frac{1}{\sqrt{2 \pi}} \int_{- \infty}^{+\infty} t^{2k} e^\frac{-t^2}{2} \, dt$$
for each positive integer $k$.

\begin{lemma} \label{lemma 3.3}
Let $k$ be a positive integer. Then, there exists some positive constant $A(k)$ such that
$$\frac{ | \{0 < n \leq N : |{\rm{Ram}}_{E/\Qq(T)} (n) - r \log \log (N) | \geq C \}|}{N} \leq \frac{2I_k \cdot (r \log \log (N))^k}{C^{2k}}$$
for each positive integer $N \geq A(k)$ and every positive real number $C$.
\end{lemma}

\begin{proof}
Given a positive integer $N$ and a positive real number $C$, let $S_{N,C}$ be the set of all integers $n \geq 1$ such that $$|{\rm{Ram}}_{E/\Qq(T)} (n) - r \log \log (N) | \geq C.$$ 
One has
$$\frac{1}{N} \sum_{\substack{{0<n \leq N} \\ n \in S_{N,C}}} 1 \leq  \frac{1}{N} \cdot \frac{1}{C^{2k}} \sum_{0 <n \leq N} ({\rm{Ram}}_{E/\Qq(T)} (n) - r \log \log (N))^{2k}.$$
By using the $2k$-th moment given in Theorem \ref{BSL}, we get 
$$\frac{1}{N} \cdot \frac{1}{C^{2k}} \sum_{0 <n \leq N} ({\rm{Ram}}_{E/\Qq(T)} (n) - r \log \log (N))^{2k} = \frac{(r \log \log (N))^k}{C^{2k}} \cdot (I_k + o(1))$$
where the $o(1)$ depends only on $k$, thus ending the proof.
\end{proof}

\begin{proof}[Proof of Corollary \ref{coro 3}]
Given a positive integer $N$, denote by $f(N)$ the number of positive integers $n \leq N$ such that $${\rm{Ram}}_{E/\Qq(T)}(n) \leq C$$ 
and by $g(N)$ the number of positive integers $n \leq N$ such that 
$$|{\rm{Ram}}_{E/\Qq(T)}(n) - r \log \log (N)| \geq |C - r \log \log (N)|.$$
If $N$ is sufficiently large (depending on $k$, $C$, and $r$), then, by Lemma \ref{lemma 3.3}, one has 
$$f(N) \leq g(N) \leq N \cdot (r \log \log (N))^k \cdot \frac{2 I_k}{(C-r \log \log(N))^{2k}},$$
as needed.
\end{proof}

\section{Proof of Theorem \ref{BSL}} \label{4}

\subsection{Proof of Theorem \ref{BSL} under an extra assumption} \label{4.1}
In this section, we prove:

\begin{proposition} \label{BSL cond}
Assume that the following condition holds:

\begin{enumerate}[label=\hspace{-1.5mm} {\rm{($*$)}}, leftmargin=*]
\item\label{star} $P_i(n) >0$ for each $i \in \{1,\dots,r\}$ and each $n \geq 1$, where the $P_i(T)$'s are
\end{enumerate}

\noindent
defined in \eqref{def:pol}.

\vspace{0.5mm}

\noindent
Then, Theorem \ref{BSL} holds.
\end{proposition}

We break the proof of Proposition \ref{BSL cond} into three parts.

\subsubsection{Approximation of ${\rm{Ram}}_{E/\Qq(T)}$ by prime divisor counting functions} \label{4.1.1}

Below, we describe the function ${\rm{Ram}}_{E/\Qq(T)}$ in terms of several prime divisor counting functions (Lemma \ref{lemma 5}).

First, we need the following lemma which summarizes our use of the classical result about ramification in specializations alluded to in \S\ref{1.4}. 

Given a prime number $p$, let $v_p$ be the $p$-adic valuation over $\Qq$ and $\Zz_{(p)}$ the localization of $\Zz$ at the prime ideal generated by $p$.

\begin{lemma} \label{lemma 3}
For each sufficiently large prime number $p$ (depending on the extension $E/\Qq(T)$) and each positive integer $n$ which is not a branch point of $E/\Qq(T)$, the following two conditions are equivalent:

\begin{enumerate}[label=\hspace{-3mm} {\rm{({\rm{\alph*}})}}, leftmargin=*]

\item\label{(a)} $p$ ramifies in the specialization $E_n/\Qq$ of $E/\Qq(T)$ at $n$,

\vspace{0.5mm}

\item\label{(b)} there is a unique index $i \in \{1,\dots,r\}$ such that $v_p(P_i(n)) > 0$ and
\end{enumerate}

\noindent
$e_i {\not \vert} v_p(P_i(n))$, where the $e_i$'s and the $P_i(T)$'s are defined in \eqref{def:ri} and \eqref{def:pol}.

\end{lemma}

\begin{proof}
For each $i \in \{1,\dots,r\}$, let $m_i(T)$ be the irreducible polynomial of $t_i$ over $\Qq$, where the $t_i$'s are defined in \eqref{def:bp}. So $P_i(T)= b_i \cdot m_i(T)$ for each index $i \in \{1,\dots,r\}$.

Below, we use the notion of meeting modulo a prime number $p$. Recall that $t$ and $t'$ in $\mathbb{P}^1(\overline{\mathbb{Q}})$ {\it{meet modulo $p$}} if there exist a number field $F$ such that $t, t' \in \mathbb{P}^1(F)$ and a valuation $v$ of $F$ lying over $v_p$ such that either $v(t) \geq 0$, $v(t') \geq 0$, and $v(t-t')>0$ or $v(1/t) \geq 0$, $v(1/t') \geq 0$, and $v((1/t)-(1/t'))>0$.

Pick a positive real number $p_0$ such that every prime number $p > p_0$ satisfies the following three conditions:

\begin{enumerate}[label=(\roman*), leftmargin=+0.8cm]

\item\label{(ii)}
$p$ does not divide $b_1 \cdots b_r$,

\item\label{(iii)}
$t_i$ and $1/{t_i}$ are integral over $\mathbb{Z}_{(p)}$ for each index $i \in \{1,\dots,r\}$ \footnote{Condition \ref{(ii)} implies that $t_1, \dots, t_r$ are integral over $\mathbb{Z}_{(p)}$.},

\item\label{(i)}
$p$ is a {\it{good prime}} in the sense of \cite[Definition 3.4]{Leg17}
\end{enumerate}

\noindent
(in particular, two distinct branch points cannot meet modulo $p$).

\noindent
Fix a prime $p > p_0$ and an integer $n \geq 1$ which is not a branch point. From condition \ref{(ii)}, one has $v_p(P_i(n)) = v_p(m_i(n))$, $i \in \{1,\dots,r\}$.

First, assume that condition \hspace{0.05mm} \ref{(b)} holds for some $i \in \{1,\dots,r\}$. Then, $v_p(m_i(n)) >0$. By the first part of \cite[Lemma 2.5]{Leg16}, the integer $n$ meets the branch point $t_i$ modulo $p$. From part (2)(a) of the {\it{Specialization Inertia Theorem}} \cite[\S3.2]{Leg17}, conditions \ref{(iii)} and \ref{(i)}, and since $v_p(m_i(n))$ is not a multiple of $e_i$, the prime number $p$ ramifies in the specialization $E_n/\Qq$ of $E/\Qq(T)$ at $n$, as needed for \hspace{0.5mm} \ref{(a)}.

Conversely, assume that $p$ ramifies in $E_n/\Qq$. From part (1) of the Specialization Inertia Theorem and condition \ref{(i)}, $n$ meets some branch point (different from $\infty$) modulo $p$. By the definition of the set $\{t_1,\dots,t_r\}$ and by the second part of \cite[Remark 2.3]{Leg16}, there is an $i \in \{1,\dots,r\}$ such that $n$ and $t_i$ meet modulo $p$. As $p$ satisfies condition \ref{(iii)}, one may apply the second part of \cite[Lemma 2.5]{Leg16} to get $v_p(P_i(n))>0$. Since $n$ meets $t_i$ modulo $p$ and $p$ satisfies conditions \ref{(iii)} and \ref{(i)}, one may apply part (2)(a) of the Specialization Inertia Theorem to get that the ramification index of each prime ideal lying over $p$ in $E_n/\Qq$ is equal to $e':=e_i / {\rm{gcd}}(e_i, v_p(P_i(n)))$. As $p$ ramifies in $E_n/\Qq$, one has $e' > 1$, i.e., $v_p(P_i(n))$ is not a multiple of $e_i$. 

It then remains to prove that an $i$ as above is unique. Assume that condition \hspace{0.55mm} \ref{(b)} holds for two indices $i \not=j \in \{1,\dots,r\}$. In particular, one has $v_p(m_i(n))>0$ and $v_p(m_j(n))>0$. By the first part of \cite[Lemma 2.5]{Leg16}, $n$ meets the two branch points $t_i$ and $t_j$ modulo $p$. Hence, there is a $\sigma$ in the absolute Galois group of $\Qq$ such that the branch points $t_i$ and $\sigma(t_j)$ meet modulo $p$. As $p$ satisfies condition \ref{(i)}, one has $t_i=\sigma(t_j)$, which contradicts the definition of the set $\{t_1,\dots,t_r\}$.
\end{proof}

Lemma \ref{lemma 3} motivates the following definition:

\begin{definition} \label{def 4}
Given two positive integers $a$ and $n$, set
$$m_a(n) = |\{ p \, : \, v_p(n) > 0 \, \, {\rm{and}} \, \, a \vert v_p(n) \}|.$$
\end{definition}

\noindent
In the special case $a=1$, we retrieve the classical function $\omega$, i.e., $\omega(n) = m_1(n)$ for each positive integer $n$.

\vspace{2mm}

In terms of Definition \ref{def 4}, Lemma \ref{lemma 3} provides the following approximation of ${\rm{Ram}}_{E/\Qq(T)}$.

\begin{lemma} \label{lemma 5}
There exists some real number $C \geq 1$ such that
$$\Big \vert {\rm{Ram}}_{E/\Qq(T)}(n) - \omega(P_E(n)) + \sum_{i=1}^r m_{e_i}(P_i(n)) \Big \vert \leq C$$
for each positive integer $n$ which is not a branch point, where the polynomial $P_E(T)$ is defined in \eqref{def:P_E}.
\end{lemma}

\noindent
As condition \ref{star} from Proposition \ref{BSL cond} holds, the integers $\omega(P_E(n))$ and $m_{e_i}(P_i(n))$, $1 \leq i \leq r$ and $n >0$, are well-defined.

\begin{proof}
By Lemma \ref{lemma 3}, there exists some real number $C \geq 1$ such that
\begin{equation} \label{4.4 1}
\Big \vert {\rm{Ram}}_{E/\Qq(T)}(n) - \sum_{i=1}^r\omega(P_i(n)) + \sum_{i=1}^r m_{e_i}(P_i(n)) \Big \vert \leq C
\end{equation}
for each positive integer $n$ which is not a branch point. 

Let $n$ be a positive integer which is not a branch point, $i \not=j \in \{1,\dots,r\}$, and $p$ a common prime divisor of $P_i(n)$ and $P_j(n)$. Assume that $p$ satisfies both conditions \ref{(ii)} and \ref{(i)} from the proof of Lemma \ref{lemma 3}. Then, one has $v_p(P_i(n)/b_i) > 0$ and $v_p(P_j(n)/b_j) > 0$. As explained in the last paragraph of the proof of Lemma \ref{lemma 3}, this provides that the branch points $t_i$ and $t_j$ are conjugate over $\Qq$, which cannot happen by the definition of the set $\{t_1,\dots,t_r\}$. Hence, there exists some positive real number $C'$ (not depending on $n$) such that 
\begin{equation} \label{4.4 2}
\Big|\omega(P_E(n)) -\sum_{i=1}^r \omega(P_i(n)) \Big| \leq C'.
\end{equation} 
It then remains to combine \eqref{4.4 1} and \eqref{4.4 2} to finish the proof.
\end{proof}

\subsubsection{Estimating moments} \label{4.1.2}

Let us start by estimating the moments of the functions $m_a$, $a \geq 2$.

\begin{lemma} \label{lemma 6}
Let $a$ and $k$ be two positive integers such that $a \geq 2$ and let $P(T) \in \mathbb{Z}[T]$ be a separable polynomial satisfying $P(n) >0$ for each positive integer $n$. Then, there exists some positive real number $C(P,k)$ such that $$\sum_{0<n \leq N} m_{a}^k (P(n)) \leq C(P,k) \cdot N$$
for each positive integer $N$.
\end{lemma}

\noindent
Note that Lemma \ref{lemma 6} fails in the case $a=1$ since $$\sum_{0<n \leq N} \omega(n) \sim N \cdot \log \log (N)$$ as $N$ tends to $\infty$ \cite{HR17}.

\begin{proof}
Let $N$ be a positive integer. Since $a \geq 2$, one has
\begin{equation} \label{eqboundinggeq2}
\sum_{0<n \leq N} m_{a}^k (P(n)) \leq \sum_{0<n \leq N}  \Big(\sum_{\substack{p^2 \vert P(n)}} 1 \Big)^k = \sum_{0 < n \leq N}  \sum_{\substack{(p_1,\dots,p_k) \\ {p_1^2 | P(n)} \\  {\dots} \\ {p_k^2 | P(n)} } }1.
\end{equation}
Pick two positive real numbers $\alpha$ and $\beta$ (depending only on the polynomial $P(T)$) such that $\sqrt{P(n)} \leq \alpha \cdot n^\beta$ for every positive integer $n$. By changing the order of summation in the right-hand side in (\ref{eqboundinggeq2}), we get
\begin{equation} \label{changeorder}
\sum_{0<n \leq N} m_{a}^k (P(n)) \leq \sum_{p_1 \leq \alpha \cdot N^\beta} \dots \sum_{p_k \leq \alpha \cdot N^\beta} \sum_{\substack{{ 0 < n \leq N} \\ {p_1^2 | P(n)} \\  {\dots} \\ {p_k^2 | P(n)} } } 1.
\end{equation}

Given a $k$-tuple $\underline{p}=(p_1,\dots,p_k)$ of prime numbers, let $S_{\underline{p}}$ be the set of distinct prime numbers appearing in $\underline{p}$ and set $\Pi_{\underline{p}} = \prod_{p \in S_{\underline{p}}} p$. Then, one has

\begin{equation} \label{prod}
\sum_{\substack{{ 0 < n \leq N} \\ {p_1^2 | P(n)} \\  {\dots} \\ {p_k^2 | P(n)} } } 1 = \sum_{\substack{ {0 < n \leq N} \\ {\Pi_{\underline{p}}^2 \vert P(n)}}} 1.
\end{equation}
Next, for each positive integer $M$, let $\nu(M)$ be the number of integers $m \in \{0,\dots, M -1\}$ such that $P(m) \equiv 0 \, \, {\rm{mod}} \, \, M$. Then,
\begin{equation} \label{prod2}
\sum_{\substack{ {0 < n \leq N} \\  {\Pi_{\underline{p}}^2 \vert P(n)}}} 1 \leq \nu(\Pi_{\underline{p}}^2) \cdot \frac{N}{\Pi_{\underline{p}}^2}.
\end{equation}
By the Chinese Remainder Theorem, one has
\begin{equation} \label{prod3}
\nu(\Pi_{\underline{p}}^2)=\prod_{p \in S_{\underline{p}}} \nu(p^2).
\end{equation}
Then, by \eqref{prod}, \eqref{prod2}, and \eqref{prod3}, we get
\begin{equation} \label{prod4}
\sum_{\substack{{ 0 < n \leq N} \\ {p_1^2 | P(n)} \\  {\dots} \\ {p_k^2 | P(n)} } } 1 \leq N \cdot \prod_{p \in S_{\underline{p}}} \frac{\nu(p^2)}{p^2}.
\end{equation}

Now, combine (\ref{changeorder}) and \eqref{prod4} to get
$$\sum_{0<n \leq N} m_{a}^k (P(n)) \leq N \cdot \sum_{p_1 \leq \alpha \cdot N^\beta} \dots \sum_{p_k \leq \alpha \cdot N^\beta} \prod_{p \in S_{\underline{p}}} \frac{\nu(p^2)}{p^2}$$
$$ \hspace{3cm} \leq N \cdot \sum_{m=1}^k  \binom{k}{m} \bigg(\sum_{p \leq \alpha \cdot N^\beta} \frac{\nu(p^2)}{p^2}\bigg)^m.$$
As $\nu(p^2) \leq {\rm{deg}}(P)$ for each prime $p$ not dividing the discriminant of $P(T)$, the inner series above is convergent, thus ending the proof.
\end{proof}

\begin{lemma} \label{lemma 7}
For each positive integer $k$, there exists some positive constant $C(k)$ such that
\begin{equation}
\bigg \vert \sum_{0<n \leq N} \Big({\rm{Ram}}_{E/\Qq(T)}(n) - \omega(P_E(n)) \Big)^k \bigg \vert \leq C(k) \cdot N
\label{9}
\end{equation}
for each positive integer $N$.
\end{lemma}

\begin{proof}
For each integer $n \geq 1$ which is not a branch point, set 
\begin{equation} \label{g(n)}
g(n)={\rm{Ram}}_{E/\Qq(T)}(n) - \omega(P_E(n)) + \sum_{i=1}^r m_{e_i}(P_i(n)).
\end{equation}
Denote the left-hand side in \eqref{9} by $f(N)$, $N \geq 1$. By \eqref{g(n)}, one has
\begin{equation} \label{f(N)}
f(N) \leq \sum_{0 <n \leq N} \sum_{m=0}^k |g(n)|^{k-m} \binom{k}{m} \bigg(\sum_{i=1}^r m_{e_i}(P_i(n)) \bigg)^m.
\end{equation}
Pick a real number $C \geq 1$ (depending only on $E/\Qq(T)$) such that $|g(n)| \leq C$ for each integer $n \geq 1$ (Lemma \ref{lemma 5}). Then, by \eqref{f(N)}, we get
\begin{equation}\label{f(N)2}
f(N) \leq (1+C)^k \cdot \sum_{0<n \leq N}\bigg(\sum_{i=1}^r m_{e_i}(P_i(n)) \bigg)^k.
\end{equation}
By H\"older's inequality, one has
\begin{equation} \label{Holder}
\bigg(\sum_{i=1}^r m_{e_i}(P_i(n)) \bigg)^k \leq r^{k-1} \cdot  \sum_{i=1}^r m_{e_i}^k(P_i(n))
\end{equation}
for each positive integer $n \leq N$. 
Then, combine \eqref{f(N)2} and \eqref{Holder} to get
$$f(N) \leq (1+C)^k \cdot r^{k-1} \cdot \sum_{i=1}^r \sum_{0<n \leq N} m_{e_i}^k(P_i(n)).$$
It then remains to apply Lemma \ref{lemma 6} to the polynomials $P_1(T), \dots, P_r(T)$ to finish the proof of Lemma \ref{lemma 7}.
\end{proof}

\subsubsection{Conclusion}

We can now complete the proof of Proposition \ref{BSL cond}. As condition \ref{star} has been assumed to hold, we may apply \cite[Theorem 4]{Hal56} and a classical result of Landau (see, e.g., \cite[$\S$XV.33, 1) b)]{SMC06}) to get
\begin{equation} \label{Hal}
\lim\limits_{\substack{N \to \infty }} \hspace{-0.5mm} \frac{1}{N} \hspace{-0.5mm} \sum_{0<n \leq N} \hspace{-1mm} \bigg(\frac{\omega(P_E(n)) - r \log \log(N)}{\sqrt{r \log \log (N)}} \bigg)^k \hspace{-1mm} = \hspace{-0.5mm} \frac{1}{\sqrt{2 \pi}} \int_{- \infty}^{+\infty} \hspace{-1mm} t^k e^\frac{-t^2}{2} \, dt
\end{equation}
for each integer $k \geq 1$. It then remains to combine \eqref{Hal} and Lemma ~\ref{lemma 7} to finish the proof of Proposition \ref{BSL cond}. $\hfill \square$

\subsection{Proof of Theorem \ref{BSL}} \label{4.2}

It suffices to show that condition \ref{star} from Proposition \ref{BSL cond} is redundant.

For each index $i \in \{1,\dots,r\}$, the leading coefficient $b_i$ of the polynomial $P_i(T)$ has been assumed to be positive. Hence, there exists some positive integer $\alpha$ such that $P_i(n+\alpha) > 0$ for each $i \in \{1,\dots,r\}$ and each positive integer $n$. Set $U=T-\alpha$. Then, condition \ref{star} holds for the extension $E/\Qq(U)$. Fix a positive integer $k$. Then, Proposition \ref{BSL cond} gives that 
$$\lim\limits_{\substack{N \to \infty }} \frac{1}{N} \sum_{0<n \leq N} \bigg(\frac{{\rm{Ram}}_{E/\mathbb{Q}(U)}(n) - r \log \log(N)}{\sqrt{r \log \log (N)}} \bigg)^k = \frac{1}{\sqrt{2 \pi}} \int_{- \infty}^{+\infty} t^k e^\frac{-t^2}{2} \, dt.$$

For each positive integer $n$, the specialization of the extension $E/\Qq(U)$ at $n$ and the specialization of the extension $E/\Qq(T)$ at $n+\alpha$ coincide. Hence, one has 
$$\lim\limits_{\substack{N \to \infty }} \frac{1}{N} \hspace{-1mm} \sum_{\alpha<n \leq N +\alpha} \hspace{-2mm} \bigg(\frac{{\rm{Ram}}_{E/\mathbb{Q}(T)}(n) - r \log \log(N)}{\sqrt{r \log \log (N)}} \bigg)^k = \frac{1}{\sqrt{2 \pi}} \int_{- \infty}^{+\infty} t^k e^\frac{-t^2}{2} \, dt.$$
One has
$$ \sum_{0 <n \leq \alpha} \bigg(\frac{{\rm{Ram}}_{E/\mathbb{Q}(T)}(n) - r \log \log(N)}{\sqrt{r \log \log (N)}} \bigg)^k \hspace{-1mm} = O((\log \log (N))^{k/2}), \quad N \to  \infty$$
and, by Lemma \ref{lemma 1}, one has
$$ \sum_{N<n \leq N + \alpha} \bigg(\frac{{\rm{Ram}}_{E/\mathbb{Q}(T)}(n) - r \log \log(N)}{\sqrt{r \log \log (N)}} \bigg)^k = O \bigg( \bigg(\frac{\log(N)}{\log \log(N)}\bigg)^k \bigg)$$
as $N$ tends to $\infty$. Hence, $$\lim\limits_{\substack{N \to \infty }} \frac{1}{N} \sum_{0<n \leq N} \bigg(\frac{{\rm{Ram}}_{E/\mathbb{Q}(T)}(n) - r \log \log(N)}{\sqrt{r \log \log (N)}} \bigg)^k = \frac{1}{\sqrt{2 \pi}} \int_{- \infty}^{+\infty} t^k e^\frac{-t^2}{2} \, dt,$$
as needed. $\hfill \square$

\section{A final remark on the non-Galois case} \label{5}

We conclude our paper by noticing that our results can easily be extended to the situation of arbitrary finite extensions of $\Qq(T)$. 

Let $T$ be an indeterminate and $E/\mathbb{Q}(T)$ a finite extension (which is not necessarily Galois). Denote its Galois closure by $\widehat{E}/\Qq(T)$. Note that the sets of branch points of $E/\Qq(T)$ and $\widehat{E}/\Qq(T)$ are the same.

First, we recall what are the specializations of $E/\Qq(T)$. Fix a point $t_0 \in \mathbb{P}^1(\mathbb{Q}$) which is not a branch point of $\widehat{E}/\Qq(T)$. Denote the prime ideals lying over $(T-t_0) {\Qq}[T-t_0]$ in $E/\Qq(T)$ by $\mathcal{P}_1, \dots, \mathcal{P}_s$. For each $l \in \{1,\dots,s\}$, the residue field at $\mathcal{P}_l$ is denoted by $E_{t_0,l}$ and the extension $E_{t_0,l}/\Qq$ is called a {\it{specialization}} of $E/\Qq(T)$ at $t_0$. The compositum in $\overline \Qq$ of the Galois closures of all specializations of the extension $E/\Qq(T)$ at $t_0$ is the specialization of the Galois closure $\widehat{E}/\Qq(T)$ at $t_0$.

Given an integer $n \geq 1$ which is not a branch point of $\widehat{E}/\Qq(T)$, let ${\rm{Ram}}_{E/\mathbb{Q}(T)}(n)$ be the number of prime numbers $p$ ramifying in some specialization $E_{n,l}/\Qq$ of $E/\Qq(T)$ at $n$. As $p$ ramifies in the compositum of finitely many extensions of $\Qq$ if and only if it ramifies in at least one of them
, we get ${\rm{Ram}}_{E/\mathbb{Q}(T)} \equiv {\rm{Ram}}_{\widehat{E}/\mathbb{Q}(T)}$. Then, Theorems \ref{BSL} and \ref{BSL bis} as well as their corollaries extend to the non-Galois case.

\end{document}